\documentclass{amsart}
\usepackage{enumerate}
\usepackage{color}
\usepackage[all]{xy}
\usepackage{todonotes}
\usepackage{tikz}
\usepackage[scriptsize]{caption}

\usepackage{hyperref}
\hypersetup{
    colorlinks,%
    citecolor=blue,%
    filecolor=blue,%
    linkcolor=blue,%
    urlcolor=blue
}

\usepackage[alphabetic,backrefs,lite]{amsrefs}

\newtheorem{thm}{Theorem}[section]
\newtheorem{main}{Theorem}
\newtheorem{prop}[thm]{Proposition}
\newtheorem{lem}[thm]{Lemma}

\theoremstyle{definition}

\newtheorem{defn}[thm]{Definition}
\newtheorem*{defn*}{Definition}

\newcommand{\Q}{\mathbb{Q}}

\newcommand{\pp}{\mathbb{P}}
\newcommand{\kk}{\mathbb{K}}
\newcommand{\zz}{\mathbb{Z}}
\newcommand{\ext}{\mathcal{E}xt}
\newcommand{\oo}{\mathcal{O}}
\newcommand{\dd}{\mathcal{D}}
\newcommand{\Ss}{\mathcal{S}}

\begin{document}

\title{Euler sequence for complete smooth $\kk^*$-surfaces}

\author[A.~Laface]{Antonio Laface}
\address{
Departamento de Matem\'atica,
Universidad de Concepci\'on,
Casilla 160-C,
Concepci\'on, Chile}
\email{alaface@udec.cl}

\author[M.~Melo]{Manuel Melo}
\address{
Departamento de Matem\'atica,
Universidad de Concepci\'on,
Casilla 160-C,
Concepci\'on, Chile}
\email{manuelmelo@udec.cl}

\maketitle

\begin{abstract}
In this note we introduce exact sequences of sheaves
on a complete smooth  $\kk^*$-surface without elliptic points.
These sequences are an attempt to generalize
the Euler sequence for a toric variety to  
complexity one surfaces.
As an application we show that such a surface
is rigid if and only if it is Fano.
\end{abstract}

\section*{Introduction}
In \cite[V.8.13]{Har}, we are presented with an exact sequence, namely
$$0\longrightarrow \Omega_{\pp^n}\longrightarrow\oo_{\pp^n}(-1)^{n+1}
\longrightarrow \oo_{\pp^n}\longrightarrow 0,$$
which allows the author to make several computations regarding differentials.
This sequence is named {\it Euler sequence} and in \cite{CLS} it is
generalized to any smooth toric variety $X$ coming from a fan $\Sigma$ whose
rays span the whole ambient lattice.
The exact sequence in question is
$$0\longrightarrow \Omega_{X}\longrightarrow\bigoplus_{\rho\in\Sigma(1)}\oo_{\pp^n}(-D_\rho)
\longrightarrow {\rm Pic}(X)\otimes_\zz\oo_{X}\longrightarrow 0,$$
where $D_\rho$ is the invariant divisor associated to the ray $\rho$.

This work attempts to obtain a similar
result in the case of $\kk^*$-surfaces.

We briefly recall that a $\kk^*$-surface $X$ without elliptic
points (see Definition~\ref{elliptic})
is equipped with an equivariant morphism
$\pi\colon X\to Y$, onto a smooth projective
curve $Y$, which admits two distinguished
sections called the source $F^+$ and the sink 
$F^-$ of $X$. This allows us to form the 
divisor $E_{\mathcal S}=mF-\sum_{i\in I}E_i$, where
$m$ is the number of reducible fibers of $\pi$
and $\{E_i : i\in I\}$ is the set of the prime components
of such fibers. 

Our main theorem is the following.

\begin{main}\label{main1}
Let $X$ be a complete smooth  $\kk^*$-surface without elliptic points.
There is an exact sequence of 
$\mathcal{O}_X$-modules\\
\[
 \begin{array}{l}
 \xymatrix{
  0\ar[r] & 
  \pi^*\Omega_{ Y}\otimes_{\mathcal{O}_X}\mathcal{O}_X(E_{\mathcal{S}})
  \ar[r] &
  \Omega_{X}\ar[r] & 
  \mathcal{G}\ar[r] & 
  \oo_X^{r+1}\ar[r] & 0
  }
 \end{array}
\]
where $r=|I|-m$ and $\mathcal G$
is the quotient of $\oplus_{i\in I}\oo_X(E_i)\oplus\oo_X(F^+)\oplus\oo_X(F^-)$
by the subsheaf $\oplus_{i,j}\oo_X(-E_i-E_j)$,
where the sum is taken over all the pairs
$(i,j)$ such that $E_i\cap E_j\neq\emptyset$.

\end{main}

The paper is organized as follows.
In Section~\ref{sec:1} we introduce
the language of polyhedral divisors
detailed in \cite{AH}. 
Section~\ref{sec:2} is devoted to the construction of two
exact sequences needed to prove Theorem~\ref{main1}.
Because of the method employed
to prove them, we restrict our attention to
complete smooth  $\kk^*$-surfaces without elliptic points. 
Finally, in Section~\ref{sec:3}, as an application, 
we show that the only 
rigid rational $\kk^*$-surfaces without elliptic points
are Fano. This is a partial generalization
of~\cite[Corollary 2.8]{Il} where the case of
toric surfaces is considered.

\begin{main}\label{main2}
Let $X$ be a complete smooth $\kk^*$-surface
without elliptic points. Then the following are
equivalent:
\begin{enumerate}
\item
The equality $h^1(X,T_X)=0$ holds;
\item
$X$ is toric Fano.
\end{enumerate}
\end{main}

{\em Acknowledgments}. It is a pleasure to thank
J\"urgen Hausen for useful discussions on the subject.

\section{Preliminaries}
\label{sec:1}

Let $\kk$ be an algebraically closed field of characteristic zero
and $X$ an algebraic variety over $\kk$ having an action of
$T:=(\kk^*)^n$ (which is called the $n$-dimensional torus).
This action is called \textit{effective} if the only $t\in T$,
for which $t\cdot x=x$ holds for all $x\in X$, is the identity of $T$.
\begin{defn*}
A \textit{$T$-variety} is a normal algebraic variety $X$
coming with an effective $(\kk^*)^n$-action.
The \textit{complexity} of $X$ is the difference $\dim X-n$.
\end{defn*}

These varieties admit a polyhedral description
given by  K. Altmann and J. Hausen for the affine case in
\cite{AH} and later, together with H. Su\ss,
in \cite{AHS} for the non-affine case.
In the following section, we briefly recall this construction
as well as a definition by N. Ilten and H. Su\ss~
in \cite{IS} that helps to simplify the notation.

\subsection{Polyhedral divisors}
Let $N$ be a lattice of rank $n$, and $M=\textrm{Hom}\langle N,\zz\rangle$ its dual. 
We denote by $\langle -, -\rangle\colon M\times N\to\zz$ the perfect pairing defined
by $(u,v)\mapsto \langle u,v\rangle:=u(v)$ and by
$N_\Q:=N\otimes_\zz\Q$, $M_\Q:=M\otimes_\zz\Q$
the rational vector spaces.
A \textit{polyhedron} in $N_\Q$ is an intersection of finitely many affine half spaces in $N_\Q$. If we require the supporting hyperplane of any half space to be a linear subspace, the polyhedron is called a \textit{cone}. If $\sigma$ is a cone in $N_\Q$, its \textit{dual cone} is defined as
$$\sigma^\vee:=\{u\in M_\Q:\langle u,v\rangle\ge 0\textrm{ for all }v\in N_\Q\}.$$
Let $\Delta\subseteq N_\Q$ be a polyhedron. The set
$$\sigma:=\{v\in N_\Q:tv+\Delta\subseteq\Delta,\ \forall t\in\Q\}$$
is a cone called the \textit{tailcone} of $\Delta$ and $\Delta$ is called a \textit{$\sigma$-polyhedron}.

\begin{defn} Let $Y$ be a normal variety and $\sigma$ a cone. A \textit{polyhedral divisor} on $Y$
is a formal sum
$$\dd:=\sum_P \Delta_P\otimes P,$$
where $P$ runs over all prime divisors of $Y$ and the $\Delta_P$ are all $\sigma$-polyhedrons such that $\Delta_P=\sigma$ for all but finitely many $P$. We admit the empty set as a valid
$\sigma$-polyhedron too.
\end{defn}

 Let $\mathfrak{D}:=\sum\Delta_P\otimes P$ be a polyhedral divisor on $Y$, with tailcone $\sigma$. For every $u\in\sigma^\vee$ we define the evaluation
$$\mathfrak{D}(u):=\sum_{\substack{P\subset Y \\ \Delta_P\neq\emptyset}}\min_{v\in\Delta_P}\langle u,v\rangle\otimes P\in{\rm WDiv}_{\Q}({\rm Loc}\,\dd)$$
where ${\rm Loc}\,\dd:=Y\setminus (\cup_{\Delta_P=\emptyset}P)$ is the \textit{locus} of $\dd$.

\begin{defn}\label{pp}
Let $Y$ be a normal variety. A \textit{proper polyhedral divisor}, also called a \textit{pp-divisor} is a polyhedral divisor $\mathfrak{D}$ on $Y$, such that
\begin{enumerate}[(i)]
\item $\mathfrak{D}(u)$ is Cartier and semiample for every $u\in\sigma^\vee\cap M$.
\item $\mathfrak{D}(u)$ is big for every $u\in(\textrm{relint}\,\sigma^\vee)\cap M$.
\end{enumerate}
\end{defn}

Now, let $\mathfrak{D}$ be a pp-divisor on a semiprojective (i.e. projective over some affine variety) variety $Y$, $\mathfrak{D}$ having tailcone $\sigma\subseteq N_\Q$. This defines an $M$-graded algebra
$$A(\mathfrak{D}):=\bigoplus_{u\in\sigma^\vee\cap M}\Gamma({\rm Loc}\,\dd,\oo(\dd(u))).$$
The affine scheme $X(\dd):={\rm Spec}\,A(\dd)$ comes with a natural action of ${\rm Spec}\,\kk[M]$.
Definition~\ref{pp} is mainly motivated
by the following result~\cite{AH}*{Theorem 3.1 and Theorem 3.4}.

\begin{thm}
Let $\mathcal D$ be a pp-divisor on a 
normal variety $Y$.
Then $X(\dd)$ is an affine T-variety of complexity equal to $\dim Y$.
Moreover, every affine $T$-variety arises like this.
\end{thm}
\subsection{Divisorial fans}
Non-affine $T$-varieties are obtained by gluing 
affine $T$-varieties coming from pp-divisors in
a combinatorial way as specified in Definition~\ref{glue}.

Consider two polyhedral divisors $\mathcal{D}=\sum\Delta_P\otimes P$ and $\mathcal{D'}=\sum\Delta_P'\otimes P$ on $Y$, with tailcones $\sigma$ and $\sigma'$
respectively and such that $\Delta_P\subseteq\Delta_P'$
for every $P$. We then have an inclusion
$$\bigoplus_{u\in\sigma^\vee\cap M}\Gamma({\rm Loc}\,\dd,\oo(\dd(u)))\subseteq \bigoplus_{u\in\sigma^\vee\cap M}\Gamma({\rm Loc}\,\dd,\oo(\dd'(u))),$$
which induces a morphism $X(\dd')\rightarrow X(\dd)$. We say
that $\dd'$ is a \textit{face} of $\dd$, denoted by $\dd'\prec\dd$, if this morphism is an open embedding.

\begin{defn}\label{glue}
A \textit{divisorial fan} on $Y$ is a finite set $\mathcal{S}$ of pp-divisors on $Y$
such that for every pair of divisors $\dd=\sum\Delta_P\otimes P$ and $\dd'=\sum\Delta_P'\otimes P$
in $\mathcal{S}$, we have $\dd\cap\dd'\in\mathcal{S}$ and
$\dd\succ\dd\cap\dd'\prec\dd'$, where 
$\dd\cap\dd':=\sum(\Delta_P\cap\Delta_P')\otimes P$.
\end{defn}

This definition allows us to glue affine $T$-varieties via
$$X(\dd)\longleftarrow X(\dd\cap\dd')\longrightarrow X(\dd'),$$
thus resulting in a scheme $X(\mathcal{S})$.
For the following theorem see~\cite{AHS}*{Theorem 5.3 and Theorem 5.6}.

\begin{thm}
The scheme $X(\mathcal{S})$ constructed above is a $T$-variety of complexity equal to $\dim Y$.  Every $T$-variety
can be constructed like this.
\end{thm}

\subsection{Marked fansy divisors on curves}
In this subsection $Y$ will be a smooth
projective curve.
For a polyhedral divisor
$\dd=\sum_P\Delta_P\otimes P$ on $Y$, and a point $y\in Y$, set
$$\dd_y:=\sum_{P\ni y}\Delta_P.$$
Then for a divisorial fan $\mathcal{S}$ on $Y$, we define the \textit{slice} of $\mathcal{S}$ at $y$
as $\{\dd_y:\dd\in\Ss\}$.

Now, the slices of a divisorial fan do not give enough information about
the corresponding $T$-variety.
Two divisorial fans $\mathcal{S}$ and $\mathcal{S}'$
can have the same slices, yet $X(\mathcal{S})\neq X(\mathcal{S}')$.
In \cite{IS}, this issue is fixed for the case of
complete complexity-one $T$-varieties
with the following definition.

\begin{defn}
A \textit{marked fansy divisor} on a curve $Y$ is a formal sum $\Xi=\sum_{P\in Y}\Xi_P\otimes P$
together with a fan $\Sigma$ and a subset $C\subseteq\Sigma$, such that
\begin{enumerate}[(i)]
\item $\Xi_P$ is a complete polyhedral subdivision of $N_\Q$, and ${\rm tail}\,(\Xi_P)=\Sigma$
for all $P\in Y$.
\item For $\sigma\in C$ of full dimension, $\sum\Delta_P^\sigma\otimes P$ is a pp-divisor,
where $\Delta_P^\sigma$ is the only $\sigma$-polyhedron of $\Xi_P$.
\item For $\sigma\in C$ of full dimension and $\tau\prec\sigma$, we have $\tau\in C$ if and only
if $(\sum_P \Delta_P^\sigma)\cap\tau\neq\emptyset$ .
\item If $\tau\prec\sigma$ and $\tau\in C$, then $\sigma\in C$.
\end{enumerate}
We say that the cones in $C$ are \textit{marked}.
\end{defn}

Given a complete divisorial fan $\mathcal{S}$ on a curve $Y$, we can define the marked fansy
divisor $\Xi=\sum_P\mathcal{S}_P\otimes P$ with marks on all the tailcones
of divisors $\dd\in\mathcal{S}$ having complete locus. We denote it by $\Xi(\mathcal{S})$.
The following is proved in~\cite{IS}*{Proposition 1.6}.

\begin{prop} For any marked fansy divisor $\Xi$, there exists a complete divisorial fan
$\mathcal{S}$ with $\Xi(\mathcal{S})=\Xi$. If two divisorial fans $\mathcal{S}$, $\mathcal{S}'$ satisfy $\Xi(\mathcal{S})=\Xi(\mathcal{S}')$, then $X(\mathcal{S})=X(\mathcal{S}')$.
\end{prop}

\subsection{$\kk^*$-surfaces}\label{k-surf}

We now look at the case of a $T$-variety of dimension two and complexity one. We call this
type of variety a \textit{$\kk^*$-surface}.

\begin{defn}\label{elliptic}
Let $X$ be a $\kk^*$-surface and let $x\in X$ be a fixed point
for the torus action. We say that the fixed point $x$ is:

\begin{itemize}
\item \textit{elliptic} if there is an invariant open neighborhood $U$ 
of $x$ such that this point lies in the closure of every orbit of $U$,
\item \textit{parabolic} if $x$ lies on a curve
made entirely of fixed points,
\item \textit{hyperbolic} otherwise.
\end{itemize}
\end{defn}

Before proceeding recall that
a morphism of varieties $\pi\colon X\rightarrow Y$ where $X$
is a $T$-variety is called a \textit{good quotient} if $\pi$ is affine, constant
on the orbits and the pullback $\pi^*\colon\oo_Y\rightarrow \pi_*(\oo_X)^T$
is an isomorphism.

\begin{prop}
Let $X$ be a complete $\kk^*$-surface. The following statements are equivalent.
\begin{enumerate}[(i)]
\item There exists a morphism $X\rightarrow Y$ onto a 
smooth projective curve $Y$ that is a good quotient
for the $\kk^*$-action.
\item $X$ has no elliptic fixed points.
\item $X$ is given by a marked fansy divisor without marks.
\end{enumerate}
\end{prop}
\begin{proof} We prove three implications.

{\em $(i)\Rightarrow (ii)$.} Let $x\in X$ be an elliptic fixed point. There is an open neighborhood $x\in U$ such that $x$ lies in the closure of every orbit in $U$. Therefore, there cannot
be a good quotient $X\rightarrow Y$ because the open set $U$
would be mapped to a single point; a contradiction.

{\em $(ii)\Rightarrow (iii)$.}
Assume that the marked fansy divisor defining $X$ has a mark. There is then an open affine chart
of $X$ given by a pp-divisor $\dd$ with complete locus $Y$.
The zero-degree component of $A(\dd)$ is $\Gamma(Y,\oo)\cong\kk$, meaning that
the degrees of the generators of $A(\dd)$ as an algebra are either all positive
(if ${\rm tail}\,\dd=\Q_{\ge 0}$) or all negative
(if ${\rm tail}\,\dd=\Q_{\le 0}$). In either case, we can take local coordinates
such that the origin $x_0$ lies in the closure of every orbit, i.e. $x_0$ is an elliptic fixed point.

{\em $(iii)\Rightarrow (i)$.}
Let $\Ss$ be a divisorial fan on a smooth projective
curve $Y$ such that the marked
fansy divisor for $X$ is $\Xi(\Ss)$ . Since there are no marks, each $\dd\in\Ss$ has
an affine locus, so there is a morphism
$X(\dd)\rightarrow {\rm Loc}\,\dd$ coming from the inclusion
$A(\dd)_0:=\Gamma({\rm Loc}\,\dd,\oo)\subseteq A(\dd)$.
These glue together to a morphism $\pi\colon X\rightarrow Y$ because
the completeness of $X$ implies that
$\{{\rm Loc}\,\dd:\dd\in\Ss\}$ is an affine open covering of $Y$.
Thus $\pi$ is a good quotient because $A(\dd)_0$ is precisely the subalgebra of invariants of $A(\dd)$.
\end{proof}

Let $X$ be a complete $\kk^*$-surface. 
There exist two invariant subsets $F^-\subseteq X$
and $F^+\subseteq X$, called \textit{sink} and \textit{source}
respectively, such that there is an open set $U\subseteq X$
where the closure in $X$ of every orbit in $U$ intersects
both $F^-$ and $F^+$. There are finitely many orbits
outside of $U\cup F^-\cup F^+$, that are called
the \textit{special} orbits. The source can be either
an elliptic point or an irreducible curve of parabolic points;
the same holds true for the sink. Every fixed point
outside of $F^+\cup F^-$ is hyperbolic.

Now, consider a complete smooth $\kk^*$-surface
having no elliptic points.  Denote by $E_1,\ldots, E_r$
the closures of the special orbits.
F. Orlik and P. Wagreich construct a graph
having vertex set $\{E_1,\ldots,E_r,F^+,F^-\}$
and two vertices are joined by an edge if and only if
the two corresponding curves intersect. Each vertex carries
a weight equal to the self-intersection number
of the curve that it represents.
This graph takes the following form.

\begin{figure}[h]\label{Fansy}
\begin{center}
\begin{tikzpicture}[scale=0.8]
 \newcommand{\nd}
 {\node[circle,draw, fill=black,inner sep=0pt,outer sep=2pt,minimum size=3pt]}
 \newcommand{\bk}
 {\node[circle, fill=white,inner sep=0pt,outer sep=2pt,minimum size=7pt]}
 \nd (S1) at (0,0) {}; \node[left] at (S1) {\small{$-c^-$}};
 \nd (S2) at (8,0) {}; \node[right] at (S2) {\small{$-c^+$}};
 \nd (E1) at (2,1) {}; \node[above] at (E1) {\small $-d_{1,1}$};
 \nd (E8) at (6,1) {}; \node[above] at (E8) {\small $-d_{1,s_1}$};
 \bk (E9) at (2,0) {$\vdots$}; 
 \bk (E12) at (6,0) {$\vdots$};
 \nd (E13) at (2,-1) {}; \node[below] at (E13) {\small $-d_{m,1}$};
 \nd (E16) at (6,-1) {}; \node[below] at (E16) {\small $-d_{m,s_m}$};
 \draw[-] (S1) -- (E1); \draw[-] (E1) -- (3,1);  \draw[dotted] (3.2,1) -- (4.8,1);
 \draw[-] (5,1) -- (E8); \draw[-] (E8) -- (S2);
 
\draw[-] (S1) -- (E9);  \draw[-] (E12) -- (S2);

 \draw[-] (S1) -- (E13); \draw[-] (E13) -- (3,-1);  \draw[dotted] (3.2,-1) -- (4.8,-1);
 \draw[-] (5,-1) -- (E16); \draw[-] (E16) -- (S2);
\end{tikzpicture}
\end{center}\caption{The graph of the $\kk^*$-surface}
\label{Graph}
\end{figure}
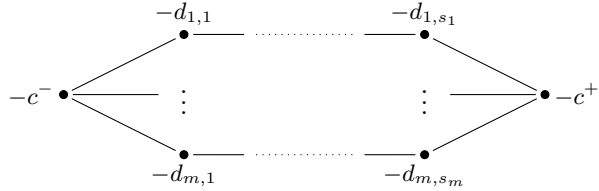

The $d_{i,j}$ are all positive and satisfy that the
Hirzebruch-Jung continued fraction
$[d_{i,1},d_{i,2},\ldots,d_{i,s_i}]$ equals 0
for every $1\le i\le m$.
On the other hand, our surface is 
given by a marked fansy divisor, without marks,
on a smooth projective curve $Y$.

\begin{figure}[h]
\begin{center}
\begin{tikzpicture}[scale=1]
 \newcommand{\nd}
 {\node[circle,draw, fill=white,inner sep=0pt,outer sep=2pt,minimum size=17pt]}
 \newcommand{\bk}
 {\node[circle, fill=white,inner sep=0pt,outer sep=2pt,minimum size=7pt]}
\bk at (-1,0) { };
\node at (7,0) {$p_1$}; \node at (7,-1.4) {$p_m$};
\node at (1,0) {\tiny $|$}; \node at (5,0) {\tiny $|$}; 
\node at (1,-1.4) {\tiny $|$}; \node at (5,-1.4) {\tiny $|$}; 
\node[above] at (1,0) {\tiny $a_{1,1}/b_{1,1}$}; \node[above] at (5,0) {\tiny $a_{1,s_1}/b_{1,s_1}$}; 
\node[above] at (1,-1.4) {\tiny $a_{m,1}/b_{m,1}$}; \node[above] at (5,-1.4) {\tiny $a_{m,s_m}/b_{m,s_m}$}; 
\node at (7,-0.6) {$\vdots$};
\node at (2.3,0.1) {$\cdot$}; \node at (3,0.1) {$\cdot$}; \node at (3.7,0.1) {$\cdot$};
\node at (2.3,-1.3) {$\cdot$}; \node at (3,-1.3) {$\cdot$}; \node at (3.7,-1.3) {$\cdot$};
 \draw[<->] (0,0) -- (6,0);  
 \draw[<->] (0,-1.4) -- (6,-1.4); 
\end{tikzpicture}
\end{center}\caption{Marked fansy divisor without marks}
\label{Fansy}
\end{figure}
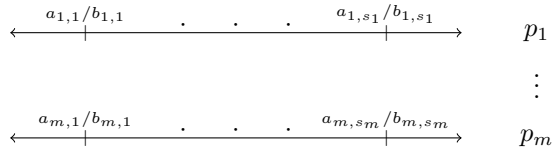

Where the smoothness of the surface
implies that $b_{i,j}a_{i,j+1}-a_{i,j}b_{i,j+1}=1$
for every $1\le i\le m$, $1\le j< s_i$, as well as
$b_{i,1}=b_{i,s_i}=1$, as shown in~\cite{Sus}*{Theorem 3.3}.
It turns out, as well, that each $b_{i,j}$ with $j>1$
equals the Hirzebruch-Jung continued fraction
$[d_{i_1},d_{i,2},\ldots,d_{i,j-1}]$.

\section{Euler sequence}
\label{sec:2}

\subsection{Euler sequence for $\kk^*$-surfaces}
Let $X$ be a complete smooth  $\kk^*$-surface
having no elliptic points given by a marked fansy
divisor $\Xi$ with no marks,
like the one depicted in Figure \ref{Fansy}.
Each fraction $a_{i,j}/b_{i,j}$ in the figure
corresponds to a divisor $E_{i,j}$ which is the
closure of a special orbit of $X$.

\begin{defn}
The multiplicity of the divisor $E_{i,j}$ is the
non-negative integer 
\[
 \mu(E_{i,j}) := b_{i,j}-1.
\]
According to the definition of the divisor
$E_{\mathcal S}$ given in the introduction
the equality 
$E_{\Ss}:=\sum_{i,j} \mu(E_{i,j})\cdot E_{i,j}$
holds, where $1\le i\le m$ and $1\le j\le s_i$.
\end{defn}

Let $\Omega_X$ and $\Omega_{Y}$
be the cotangent sheaves of $X$ and $Y$
respectively. 
As in Section \ref{k-surf}, let $F^-$ and $F^+$
denote the source and sink of $X$.
Let $Z\subseteq X$ be the set of hyperbolic
fixed points of $X$. In what follows
we will call $\mathcal{F}_X$
the sheaf $\mathcal{O}_{F^-}\oplus\mathcal{O}_{F^+}\oplus\mathcal{O}_Z$.

\begin{lem}\label{PreEuler}
Let $X$ be a complete smooth  $\kk^*$-surface without elliptic points.
There is an exact sequence of 
$\mathcal{O}_X$-modules\\
\[
 \begin{array}{l}
 \xymatrix{
  0\ar[r] & 
  \pi^*\Omega_{Y}\otimes_{\mathcal{O}_X}\mathcal{O}_X(E_{\mathcal{S}})
  \ar[r]^-{\imath} &
  \Omega_{X}\ar[r]^-\alpha & 
  \mathcal{O}_X\ar[r] & 
  \mathcal{F}_X\ar[r] & 0
  }
 \end{array}
\]

where $\alpha$ is defined by
$fdz\mapsto \deg(z)fz$, for any
homogeneous local coordinate $z$
with respect to the $\zz$-grading
of $\mathcal{O}_X$ induced by the $\kk^*$-action
and $\imath$ is defined by $dt\otimes f\mapsto fdt$,
where $t$ is the pull-back of a local
coordinate on $Y$.

\end{lem}
\begin{proof}
Let $\Xi$ be a fansy divisor describing $X$.
Each affine chart of $X$, or an intersection
of them, is given by a polyhedron
$\Delta$ on some slice of $\Xi$. In other words,
it is given by the pp-divisor 
$$\dd=\Delta\otimes p+\sum_{\mathcal{P}-\{p\}}\emptyset\otimes p,$$
where $\mathcal{P}$ is the set of points
of $Y$ with non-trivial slices. We analyze each possible
$\Delta$ separately.
\vspace{2mm}

{\em Case 1}.
$\Delta=[a_1/b_1,a_2/b_2]$ with $b_1b_2\neq 0$.
In this case $\mathcal{D}$ defines an open
affine subset $X_{\mathcal{D}}$ of $X$ which
is the spectrum of the algebra
\[
 \bigoplus_{u\in\zz}
 \Gamma({\rm loc}(\mathcal{D}),\mathcal{O}(\mathcal{D}(u)))
 =
 S^{-1}\kk[x^{a_2}\chi^{-b_2},x^{-a_1}\chi^{b_1}]
 \cong S^{-1}\kk[z,w] =: R,
\]
where $x$ is a regular function of ${\rm loc}(\mathcal{D})$
which has a simple zero at $p$ and is non-zero
at ${\rm loc}(\mathcal{D})-\{p\}$, while
$S\subseteq\kk[x^{a_2}\chi^{-b_2},x^{-a_1}\chi^{b_1}]$ 
is the multiplicative system defined by degree zero
homogeneous polynomials
which do not vanish on ${\rm loc}(\mathcal{D})$.
The first equality is due to our assumption
on $\mathcal{D}$.
We have an exact sequence
\[
 \xymatrix{
  Rdz\oplus Rdw\cong \Omega_{R}\ar[r]^-\alpha
  & R\ar[r] & R/I \ar[r] & 0
 }
\]
where $I\subseteq R$ is the ideal $\langle {\rm deg}(z)z,{\rm deg}(w)w\rangle$.
The restriction of the quotient map
$\pi\colon X\to Y$ to the open
subset $X_{\mathcal{D}}$ is defined
by the inclusion $\kk[x]\subseteq R$.
Since $x = z^{\deg(w)}w^{\deg(z)} = z^{b_1}w^{b_2}$,
the curve $\pi^{-1}(p)\cap X_{\mathcal{D}}$
has two irreducible components which
are vertical curves intersecting at the
fixed point $q\in Z$ of local coordinates
$z=w=0$. Thus $R/I$ defines the skyscraper
sheaf $\mathcal{O}_q$ and we get the first
exact sequence from
$\mathcal{F}|_{X_{\mathcal{D}}}
\cong \mathcal{O}_q$.
The sheaf $\pi^*\Omega_{Y}$
is locally generated by $dx=z^{b_1-1}w^{b_2-1}(b_1wdz+b_2zdw)$,
thus we have the desired isomorphism
\[
 \pi^*\Omega_{Y}\otimes_{\mathcal{O}_X}\mathcal{O}_X((b_1-1)E_1+(b_2-1)E_2)
 |_{X_{\mathcal{D}}}
 \to
 \ker(\alpha)|_{X_{\mathcal{D}}},
\]
where  for $i=1,2$, the divisor $E_i$ is the one
associated to the fraction $a_i/b_i$, as
explained in the beginning of this section.

{\em Case 2}.
$\Delta=[a_1,\infty)$.
In this case $\mathcal{D}$ defines an open
affine subset $X_{\mathcal{D}}$ of $X$ which
is the spectrum of the algebra
\[
 \bigoplus_{u\in\zz_{\ge 0}}
 \Gamma({\rm loc}(\mathcal{D}),\mathcal{O}(\mathcal{D}(u)))
 =
 S^{-1}\kk[x,x^{-a_1}\chi]
 \cong S^{-1}\kk[z,w] =: R,
\]
where $x$ and $S$ are defined in a similar
way as in the first case. 
Since $x = z$, the curve $\pi^{-1}(p)\cap X_{\mathcal{D}}$
has one irreducible component which
is a vertical curve intersecting $F^+$
at one point. Again we got an 
exact sequence as above and observe that
now $I=\langle w\rangle$. In this case
$R/I$ defines the sheaf $\mathcal{O}_{F^+}|_{X_{\mathcal{D}}}$
and we get the first exact sequence from
$\mathcal{F}|_{X_{\mathcal{D}}}
\cong \mathcal{O}_{F^+}|_{X_{\mathcal{D}}}$.
The sheaf $\pi^*\Omega_{Y}$
is locally generated by $dx=dz$,
thus we have an isomorphism
\[
 \pi^*\Omega_{Y}|_{X_{\mathcal{D}}}
 \to
 \ker(\alpha)|_{X_{\mathcal{D}}}.
\]

{\em Case 3}.
$\Delta=(-\infty,a_2]$. This is similar to the 
previous case and we omit the details.

{\em Case 4}.
$\Delta=\{a/b\}$.
In this case $\mathcal{D}$ defines an open
affine subset $X_{\mathcal{D}}$ of $X$ which
is the spectrum of the algebra
\[
 \bigoplus_{u\in\zz}
 \Gamma({\rm loc}(\mathcal{D}),\mathcal{O}(\mathcal{D}(u)))
 =
 S^{-1}\kk[x^k\chi^{-l},(x^{-a}\chi^{b})^{\pm 1}]
 \cong S^{-1}\kk[z,w^{\pm 1}] =: R,
\]
where $x$ and $S$ are defined in a similar
way as in the first case and $c,d$ are integers
such that $bk-la=1$.
Since $x = z^{\deg(w)}w^{\deg(z)} = z^{b}w^{l}$,
the curve $\pi^{-1}(p)\cap X_{\mathcal{D}}$
has an irreducible component which
is a vertical curve which has empty intersection
with $F^+\cup F^-\cup Z$.
Again we get an exact sequence as in the 
first case with $I=\langle z,w^{\pm 1}\rangle=R$. In this case
$R/I=0$ and we get the first exact sequence from
$\mathcal{F}|_{X_{\mathcal{D}}} = 0$.
The sheaf $\pi^*\Omega_{Y}$
is locally generated by $dx=z^{b-1}w^{l-1}(bwdz+lzdw)$,
thus we have an isomorphism (since $w$ is a unit
in this chart)
\[
 \pi^*\Omega_{Y}\otimes_{\mathcal{O}_X}\mathcal{O}_X((b-1)E)|_{X_{\mathcal{D}}}
 \to
 \ker(\alpha)|_{X_{\mathcal{D}}},
\]
where $E$ is the divisor associated to the fraction $a/b$.

\end{proof}
\vspace{4mm} 

We define the sheaf
$\mathcal{Q}_\alpha :=\Omega_X/\ker(\alpha).$

Now, maintaining the same hypothesis and notation as above, 
we are ready to prove part two of Theorem~\ref{PreEuler}.

\begin{lem}\label{Euler}
There is a short exact sequence of $\mathcal{O}_X$-modules
$$0\longrightarrow\mathcal{Q}_\alpha\longrightarrow\mathcal{G}\longrightarrow \mathbb{Z}^{r+1}\otimes\mathcal{O}_X\longrightarrow 0,$$
where $$\mathcal{G}=
\oo(-F^-)\oplus\oo(-F^+)\oplus\big(\ (\oplus_{i,j}\oo(-E_{i,j}))\,\big/
(\,\oplus_{i=1}^{m}\oplus_{j=1}^{s_i}\oo(-E_{i,j}-E_{i,j+1}))\ \big).$$
\end{lem}
\begin{proof}
Let us consider the diagram
\[
\xymatrix@R=3mm@C=3mm
{&&0\ar@{.>}[dd]&&0\ar[dd]&&0\ar[dd]&&\\
&&&&&&&&\\
0\ar[rr]&&\mathcal{Q}_\alpha\ar[rr]\ar@{.>}[dd]
&&\mathcal{O}_X\ar[rr]\ar[dd]&&\mathcal{F}_X\ar[rr]\ar@{=}[dd]&&0\\
&&&&&&&&\\
0\ar[rr]&&\mathcal{G}\ar[rr]\ar@{.>}[dd]
&&\mathbb{Z}^{r+2}\otimes\oo_X\ar[rr]\ar[dd]&&\mathcal{F}_X\ar[rr]\ar[dd]&&0\\
&&&&&&&&\\
0\ar[rr]&&\mathbb{Z}^{r+1}\otimes\oo_X\ar@{=}[rr]\ar@{.>}[dd]&&\mathbb{Z}^{r+1}\otimes\oo_X\ar[rr]\ar[dd]&&0&&\\
&&&&&&&&\\
&&0&&0&&&&}
\]
where the top row comes from Lemma~\ref{PreEuler}.
The middle columns is the direct sum of the fundamental short exact sequences
$$0\longrightarrow \oo(-F^\pm)\longrightarrow \oo_X\longrightarrow \oo_{F^\pm} \longrightarrow 0$$
together with the following exact sequences (cf. \cite{Be}) for each
$E_i\cap E_j=p\in Z$
$$0\longrightarrow\oo(-E_{i}-E_{j})\longrightarrow\oo(-E_{i})\oplus\oo(-E_{j})\longrightarrow\oo_X\longrightarrow\oo_{p}\longrightarrow 0$$
where we replace the first two sheaves with their quotient to obtain short sequences.\\
The middle column is simply the exact sequence of modules
\[
\xymatrix@R=12pt@C=10pt
{0\ar[rr]&&\mathbb{Z}\ar^-{\tiny (1,\ldots,1)}[rr]&&\mathbb{Z}^{r+2}\ar[rr]&&\mathbb{Z}^{r+1}\ar[rr]&&0}
\]
after tensoring by $\oo_X$.
The exactness of the sequences, together with
the commutativity of both squares (easy to check),
ensures the existence of an exact sequence
on the left column.
\end{proof}

We can now prove the first theorem stated in the introduction.

\begin{proof}[Proof of Theorem \ref{main1}]
It is direct from Lemmas \ref{PreEuler} and \ref{Euler} since ${\rm im}(\alpha)$
in $\oo_X$ is isomorphic to $\mathcal{Q}_\alpha$.
\end{proof}

\begin{prop}
The following holds: $\ext^1(\mathcal{Q}_\alpha,\mathcal{O}_X)
\cong \mathcal{O}_Z$.
\end{prop}
\begin{proof}
By the definition of $\mathcal{Q}_\alpha$,
Lemma~\ref{PreEuler} and the long
exact sequence for ext sheaves we have
$\ext^1(\mathcal{Q}_\alpha,\mathcal{O}_X)
\cong \ext^2(\mathcal{F},\mathcal{O}_X)$.
Since the functor $\ext^i$ commutes with
finite direct sums, it is enough to show 
that $\ext^2(\mathcal{O}_{F^{\pm}},\mathcal{O}_X)=0$
and $\ext^2(\mathcal{O}_{p_i},\mathcal{O}_X)\cong
\mathcal{O}_{p}$ for any $p\in Z$.
Taking the long exact $\ext$-sequence 
of the exact sequence of sheaves
\[
 \xymatrix{
  0\ar[r] & 
  \mathcal{O}_X(-F^{\pm})\ar[r] &
  \mathcal{O}_X\ar[r] &
  \mathcal{O}_{F^{\pm}}\ar[r] & 
  0 
 }
\]
and using $\ext^i(\mathcal{O}_X,\mathcal{O}_X)=0$
for any $i>0$, by~\cite[Pro. III.6.3(b)]{Har},
we get $\ext^1(\mathcal{O}_X(-F^{\pm}),\mathcal{O}_X)
\cong \ext^2(\mathcal{O}_{F^{\pm}},\mathcal{O}_X)$. 
By~\cite[Pro. III.6.7]{Har} we conclude that the
these sheaves are the zero sheaf, proving the first
vanishing. To prove the second isomorphism
observe that for each $p\in Z$ lying in
the intersection $E_i\cap E_j$ we have
the following exact sequence of sheaves
~\cite{Be}
\[
 \xymatrix@1{
  0\ar[r] & 
  \mathcal{O}_X(-E_{i}-E_{j})\ar[r] &
  \mathcal{O}_X(-E_{i})\oplus\mathcal{O}_X(-E_{j})\ar[r]^-\varphi &
  \mathcal{O}_X\ar[r] &
  \mathcal{O}_{p}\ar[r] &
  0.
 }
\]
Denoting by $\mathcal{N}$ the quotient
sheaf $\mathcal{O}_X(-E_{i})\oplus\mathcal{O}_X(-E_{j})/\mathcal{O}_X(-E_{i}-E_{j})$
we deduce $\ext^1(\mathcal{N},\mathcal{O}_X)
\cong\ext^2(\mathcal{O}_p,\mathcal{O}_X)$
and the fact that $\ext^1(\mathcal{N},\mathcal{O}_X)$
is the cokernel of the map
$\mathcal{O}_X(E_{i})\oplus\mathcal{O}_X(E_{j})
\to \mathcal{O}_X(E_{i}+E_{j})$
induced by $\varphi$ taking tensor 
product with $\mathcal{O}_X(E_{i}+E_{j})$.
This proves the statement.
\end{proof}

\section{Aplications}
\label{sec:3}

\begin{lem}
\label{blow-up}
Let $\varphi\colon \tilde X\to X$ be the blow-up
of a smooth projective variety at a point $p\in X$.
Then $h^1(\tilde X,T_{\tilde X})\geq h^1(X,T_X)$.
\end{lem}
\begin{proof}
Since $\varphi$ is a blow-up it follows that
$R^i\varphi_* T_{\tilde X}$ vanishes for any $i>0$.
Thus the equality $h^i(\tilde X,T_{\tilde X})
= h^i(X,\varphi_* T_{\tilde X})$ holds for any $i$
by~\cite[Exercise III.8.1]{Har} and we conclude 
by the following exact sequence of sheaves
\[
 \xymatrix@1
 {
  0\ar[r] & \varphi_*T_{\tilde X}\ar[r]
  & T_X\ar[r] & T_p\ar[r] & 0.
 }
\]
\end{proof}

\begin{proof}[Proof of Theorem \ref{main2}]
We begin by showing $(1)\Rightarrow (2)$.
Consider the good quotient map $\pi\colon X\rightarrow Y$.
Assume first that the curve $Y$ has positive genus. If
$\pi$ has only irreducible fibers,
$X$ is a ruled surface so by \cite[Theorem 4]{Se}
we have $h^1(X,T_X)>0$. This still holds if
there are reducible fibers, by Lemma~\ref{blow-up}, 
because $X$ would be
a blow-up of one of such ruled surfaces. Thus, $Y$
must necessarily be rational.

We show now that $X$ contains no invariant rational curves $C$ with $C^2=-n\leq -2$.
Suppose such a curve exists.
From Lemma \ref{PreEuler}, after tensoring by $\oo(K_X+C)$, we have an exact sequence
$$0\rightarrow \pi^*(\Omega_\pp^1)\otimes\oo(E_\Ss+K_X+C)\rightarrow \Omega_X(K_X+C)
\rightarrow {\rm im}(\alpha)\otimes\oo(K_X+C) \rightarrow 0.$$
Let us compute some cohomology groups for these sheaves.
Assume that $K_X+C$ is linearly equivalent to an effective divisor. From the genus formula we have
$$(K_X+C)\cdot C = 2{\rm g}(C) - 2 = -2 < 0,$$
so by applying~\cite[Proposition V.1.1.2]{Libro} we see that $C$ must be in the base locus of
$|K_X+C|$, meaning $K_X+C \sim C + E'$ for some effective divisor $E'$.
This would imply that $K_X$ is linearly equivalent to an effective
divisor, a contradiction because $X$ is rational and smooth. 
Thus $h^0(X,K_X+C)=0$.
Since ${\rm im}(\alpha)\otimes\oo(C+K)$ injects into $\oo(C+K)$, then also
$$h^0(X,{\rm im}(\alpha)\otimes\oo(C+K))=0.$$

If $F$ is a general fiber of $\pi$, the genus formula yields $F\cdot K_X= -2$. The product
$F\cdot C$ equals at most 1 (where the equality holds
if $C$ is the source or sink curve), so
$$F\cdot (-2F+E_\Ss +K_X+C)=F\cdot K_X + F\cdot C<0.$$
Then $h^0(X,\pi^*(\Omega_\pp^1)\otimes\oo(E_\Ss+K_X+C))=h^0(X,-2F+E_\Ss +K_X+C)=0.$
Going back to the exact sequence, we can now deduce that $h^0(X,\Omega_X(K_X+C))=0,$ and due to
Serre's duality we conclude $h^2(X,T_X(-C))=0$.
\vspace{3mm}

Consider now the exact sequence of sheaves
$$0\longrightarrow \oo_X(-C)\longrightarrow \oo_X\longrightarrow \oo_C\longrightarrow 0.$$
Tensoring by $T_X$ gives a new exact sequence
$$0\longrightarrow T_X(-C)\longrightarrow T_X\longrightarrow T_X|_C\longrightarrow 0.$$
From the vanishing at $H^2$ shown above, there is a surjection $H^1(X,T_X)\rightarrow H^1(X,T_X|_C)$, so
it suffices to show that  $h^1(X,T_X|_C)\neq 0$ to prove the non-existence of this curve $C$, but
this comes directly from the exact sequence
$$0\longrightarrow T_C\longrightarrow T_X|_C\longrightarrow N_{C|X}\longrightarrow 0$$
and the fact that $h^1(X,T_C)=h^2(X,T_C)=0$ and $h^1(X,N_{C|X})=n-1$.

We showed that any invariant rational curve of
$X$ has self-intersection $\geq -1$. Since the 
classes of these curves generate the Mori cone 
of $X$ (see~\cite{Libro}) we conclude that $-K_X$
is ample and thus $X$ is del Pezzo.
Moreover by~\cite[Proposition 5.9]{Hu}
del Pezzo $\kk^*$-surfaces without elliptic
fixed points are toric.

The proof of $(2)\Rightarrow (1)$ is a consequence
of~\cite[Corollary 2.8]{Il}.

\end{proof}

\end{document}